\documentclass[12pt,reqno]{amsart} 

\usepackage{amssymb,latexsym}
\usepackage{cite} 

\usepackage{fontenc}
\usepackage{times}
\usepackage{cases}

\usepackage{amsfonts,amssymb,epsfig}
\usepackage{amscd}
\usepackage{color}
 \usepackage{hyperref}

\usepackage{textcomp}
\usepackage{mathrsfs}

\usepackage[height=200mm,width=140mm]{geometry} 
 \theoremstyle{exec}
 \newtheorem{exec}{Counter-Example}[section]
 \newtheorem{definition}[exec]{Definition}
 \newtheorem{rem}[exec]{Remark}
 \newtheorem{theo}[exec]{Theorem}
 \newtheorem{cor}[exec]{Corollary}
 \newtheorem{exe}[exec]{Example}
\newtheorem{lemma}[exec]{Lemma}
\newtheorem{pro}[exec]{Proposition}

\numberwithin{equation}{section} 

\begin{document}
\title[Conformal deformations preserving
the Finslerian $R$-Einstein criterion]{Conformal deformations preserving
the Finslerian $R$-Einstein criterion} 
\author{Serge Degla}

\address{Ecole Normale Sup\'erieure de Natitingou\\ P.~O. Box 72 \\ 
 Natitingou\\ B\'enin}

 \email{sdegla@imsp-uac.org}


\author{Gilbert Nibaruta}
\address{Universit\'e d'Abomey-Calavi\\ Institut de Math\'ematiques et de Sciences Physiques \\ B.P. 613\\
  Porto-Novo\\ B\'enin}

\email{nibarutag@gmail.com}

\author{L\'eonard Todjihounde}
\address{Universit\'e d'Abomey-Calavi\\ Institut de Math\'ematiques et de Sciences Physiques \\ B.P. 613\\
  Porto-Novo\\ B\'enin}

\email{leonardt@imsp-uac.org}
\begin{abstract}
Given a Finslerian metric $F$ on a $C^4$-manifold, conformal deformations of 
$F$ preserving the $R$-Einstein criterion are presented. In particular, locally conformal invariance between two 
Finslerian $R$-Einstein metrics is characterized.
\end{abstract} 
\subjclass[2010]{53C60, 58B20}

\keywords{Finslerian metrics; Einstein condition; Conformal deformations}

\maketitle

\section{Introduction}\label{Section1}
 Let $F$ be a Finslerian metric on an $n$-dimensional manifold $M$.  
One of the most interesting problems is to study the conformal invariance of 
some important geometric quantities associated with $F$ \cite{biBasco}. For example, we have the Liouville transformation that is a conformal 
 deformation which preserves the Finslerian Ricci tensor \cite{biNibaruta1}. It is known that 
 Einstein metrics play an important role in conformal geometry (see \cite{biBesse,biBrinkmann2, biRademacher2}). However, 
 little work has been done on the conformal deformations 
 between Einstein metrics of type Finslerian. 
 In 2013, Zhang Xiao-ling \cite{biZhang} has shown that conformal deformations between Einstein metrics of type Randers
 must be homothetic.
 
 The main objective of the present work is to study the conformal deformations preserving the $R$-Einstein criterion. 
In particular, we classify locally conformal deformations between two Finslerian $R$-Einstein metrics.

 The sections \ref{Section2} and \ref{dhqhdhq} review the main notions on global Finslerian $R$-Einstein spaces introduced in \cite{biNibaruta4}.
 In the Section \ref{adgbjah}, we prove our main results given by the Proposition \ref{cor10'} 
 and, established in the Theorem \ref{klwlefh} and the Corollary \ref{uieleied}. This is concluded by an example given in Section \ref{section3.gg}.
  
   \section{Preliminaries}\label{Section2}
   Throughout this work, 
   all manifolds and mappings are supposed to be differentiable of classe $C^4$. 
 Let $M$ be an $n-$dimensional manifold. We denote by $T_xM$ the tangent space at 
 $x\in M$ and by $TM:=\bigcup_{x\in M}T_xM$ the tangent bundle of $M$.
 Set $\mathring{T}M=TM\backslash\{0\}$ and $\pi:TM\longrightarrow M:\pi(x,y)\longmapsto x$ the natural projection. 
 Let $(x^i)_{i=1,...,n}$ be a local coordinate 
 on an open subset $U$ of $M$ and $(x^i,y^i)_{i=1,...,n}$ be the local coordinate 
 on $\pi^{-1}(U)\subset TM$. We have the coordinate bases  
  $\{\frac{\partial}{\partial x^i}\}_{i=1,...,n}$ and $\{dx^i\}_{i=1,...,n}$ respectively, for the tangent bundle 
  $TM$ and cotangent bundle $T^*M$.
  We use Einstein summation convention.
  \begin{definition}\label{defi1}
A \textit{Finslerian metric} 
 on $M$ is a function $F:TM\longrightarrow [0,\infty)$ with the following properties:
 \begin{itemize}
  \item [(1)] $F$ is $C^{\infty}$ on the entire slit tangent bundle $\mathring{T}M$,
  \item [(2)] $F$ is positively $1$-homogeneous on the fibers of $TM$, that is \\
  $\forall c>0,~F(x,cy)=cF(x,y),$
  \item [(3)] the Hessian matrix $(g_{ij}(x,y))_{1\leq i,j\leq n}$ with elements
  \begin{eqnarray}\label{1}
   g_{ij}(x,y):=\frac{1}{2}\frac{\partial^2F^2(x,y)}{\partial y^i\partial y^j}
  \end{eqnarray}
  is positive definite at every point $(x,y)$ of $\mathring{T}M$.
 \end{itemize}
 \end{definition}
\begin{rem}
$
 F(x,y)\neq 0 \text{   for all   } x\in M \text{   and for every   } y\in T_xM\backslash\{0\}.
 $
\end{rem}

Consider the tangent mapping $\pi_*$  of the submersion 
 $\pi:\mathring{T}M \longrightarrow M$.
 The vertical subspace of $T\mathring{T}M$ is defined by 
 $\mathcal{V}:=ker({\pi_*})$ 
 which is locally spanned by the set $\{F\frac{\partial}{\partial y^i}, 1\leq i\leq n\}$,
  on each $\pi^{-1}(U)\subset \mathring{T}M$.
  
  An horizontal subspace $\mathcal{H}$ of $T\mathring{T}M$ is by definition any complementary to 
 $\mathcal{V}$. The bundles $\mathcal{H}$ and $\mathcal{V}$ give a smooth splitting 
 \begin{eqnarray}\label{decomposition}
T\mathring{T}M=\mathcal{H}\oplus\mathcal{V}.  
 \end{eqnarray}
An Ehresmann connection is a selection of a horizontal subspace $\mathcal{H}$ of $T\mathring{T}M$.
As explain in \cite{biRatiba}, $\mathcal{H}$ can be canonically defined from the geodesics equation.

\begin{definition}\label{defi2} Let $\pi:\mathring{T}M\longrightarrow M$ be the 
submersion.
\begin{itemize}
 \item [(1)]An Finslerian Ehresmann connection of $\pi$
is the subbundle $\mathcal{H}$ of $T\mathring{T}M$ given by 
 \begin{eqnarray}
  \mathcal{H}:=ker \theta,
 \end{eqnarray}
where $\theta:T\mathring{T}M\longrightarrow \pi^*TM$ is the bundle morphism defined 
by 
    \begin{eqnarray}
   \label{03b}
   \theta=\frac{\partial}{\partial x^i}\otimes \frac{1}{F}(dy^i+N_j^idx^j).
  \end{eqnarray}
\item [(2)]The form $\theta:T\mathring{T}M\longrightarrow \pi^*TM$ induces a linear map 
\begin{eqnarray}
 \theta|_{(x,y)}:T_{(x,y)}\mathring{T}M\longrightarrow T_xM,
\end{eqnarray} 
  for each point $(x,y)\in \mathring{T}M$; where $x=\pi(x,y)$. \\
 The vertical lift of a section $\xi$ of $\pi^*TM$ is a unique section $\textbf{v}(\xi)$ of $T\mathring{T}M$ 
such that for every $(x,y)\in \mathring{T}M$, 
\begin{eqnarray}
 \pi_* (\textbf{v}(\xi))|_{(x,y)}=0_{(x,y)}\text{     and     }\theta (\textbf{v}(\xi))|_{(x,y)}=\xi_{(x,y)}.
\end{eqnarray}
\item [(3)]The differential projection $\pi_*:T\mathring{T}M\longrightarrow \pi^*TM$ induces a linear map 
\begin{eqnarray}
 \pi_*|_{(x,y)}:T_{(x,y)}\mathring{T}M\longrightarrow T_xM,
\end{eqnarray} 
  for each point $(x,y)\in \mathring{T}M$; where $x=\pi(x,y)$. \\
 The horizontal lift of a section $\xi$ of $\pi^*TM$ is a unique section $\textbf{h}(\xi)$ of $T\mathring{T}M$ 
such that for every $(x,y)\in \mathring{T}M$,
\begin{eqnarray}
 \pi_* (\textbf{h}(\xi))|_{(x,y)}=\xi_{(x,y)}\text{     and     }\theta (\textbf{h}(\xi))|_{(x,y)}=0_{(x,y)}.
\end{eqnarray}
\end{itemize}
 \end{definition}

  \begin{definition}\label{defi3b} 
  A tensor field $T$ of type $(p_1,p_2;q)$ on $(M,F)$ is a mapping 
  $$T:\underbrace{\pi^*TM\otimes...\otimes\pi^*TM}_{p_1-times}
  \otimes \underbrace{T\mathring{T}M\otimes...\otimes T\mathring{T}M}_{p_2-times}
  \longrightarrow \bigotimes^q\pi^*TM,$$
  $\big(p_1,p_2$ and $q\in\mathbb{N}\big)$ which is $C^{\infty}(\mathring{T}M,\mathbb{R})$
  -linear in each argument.
 \end{definition}
 \begin{rem}
  In a local chart, $$T=T_{i_1...i_{p_1}j_1...j_{p_2}}^{k_1...k_q}\partial_{k_1}\otimes ... \otimes\partial_{k_q}
  \otimes dx^{i_1}\otimes...\otimes dx^{i_{p_1}}\otimes \varepsilon^{j_{1}}\otimes ... \otimes\varepsilon^{j_{p_2}}$$
  where $\{\partial_{k_r}:=\frac{\partial}{\partial x^{k_r}}\}_{r=1,...,q}$ and $\{\varepsilon^{j_{s}}\}_{s=1,...,p_2}$
  are respectively the basis sections for $\pi^*TM$ and $T\mathring{T}M$.
 \end{rem}

  \begin{exe}
   \begin{itemize}
   \item[(1)] A vector field $X$ on $\mathring{T}M$ is of type $(0,1;0)$.
    \item [(2)]A section $\xi$ of $\pi^*TM$ is of type $(1,0;0)$.
   \end{itemize}
  \end{exe}
The following lemma defines the Chern connection on $\pi^*TM$.
\begin{lemma}\cite{biRatiba}\label{lem1}
  Let $(M,F)$ be a Finslerian manifold and $g$ its fundamental tensor. 
 There exist a unique linear connection $\nabla$ 
 on the bundle $\pi^*TM$ such that, for all 
 $X,Y\in \chi(\mathring{T}M)$ and for every $\xi, \eta\in\Gamma(\pi^*TM)$, one has the following properties:
 \begin{itemize}
  \item [(i)] Symmetry: $$\nabla_X\pi_*Y-\nabla_Y\pi_*X=\pi_*[X,Y],$$
  \item [(ii)] Almost $g$-compatibility:
  $$X(g(\xi,\eta))=g(\nabla_X\xi,\eta)+g(\xi,\nabla_X\eta)+2\mathcal{A}(\theta(X),\xi,\eta),$$ 
  where $\mathcal{A}:=\frac{F}{2}\frac{\partial g_{ij}}{\partial y^k}dx^i\otimes dx^j\otimes dx^k$ is the Cartan tensor.
 \end{itemize}
 \end{lemma}

 One has
 $\nabla_{\frac{\delta}{\delta x^j}}\frac{\partial}{\partial x^k}=\Gamma_{jk}^i\frac{\partial}{\partial x^i}, \text{   }
\Gamma_{jk}^i:=\frac{1}{2}g^{il}\left(\frac{\delta g_{jl}}{\delta x^k}+\frac{\delta g_{lk}}{\delta x^j}
 -\frac{\delta g_{jk}}{\delta x^l}\right)$
 where\\
  $\left\{\frac{\delta}{\delta x^i}:=\frac{\partial}{\partial x^i}-N_i^j\frac{\partial}{\partial y^j}
  =\textbf{h}(\frac{\partial}{\partial x^i})\right\}_{i=1,...,n}\text{     with     }
 N_j^i=\Gamma_{jk}^iy^k.$
\section{Finslerian $R$-Einstein metrics}\label{dhqhdhq}
\subsection{First curvatures $R$ associated with the Chern}
\begin{definition}
 The full curvature of a linear connection $\overline{\nabla}$
 on the vector bundle $\pi^*TM$ over the manifold $\mathring{T}M$ is the application
$$ \phi:\begin{matrix}
 \chi(\mathring{T}M)\times\chi(\mathring{T}M)\times\Gamma(\pi^*TM)&\to&\Gamma(\pi^*TM)\\
 (X,Y,\xi)&\mapsto
 &\phi(X,Y)\xi=\overline{\nabla}_X\overline{\nabla}_Y\xi-\overline{\nabla}_Y\overline{\nabla}_X\xi-\overline{\nabla}_{[X,Y]}\xi.
 \end{matrix}$$
 \end{definition}
 By the relation (\ref{decomposition}), we have
  \begin{eqnarray}
   \overline{\nabla}_X=\overline{\nabla}_{\hat{X}}+\overline{\nabla}_{\check{X}},\label{Bsndn1}
  \end{eqnarray}
where $X=\hat{X}+\check{X}$ with $\hat{X}\in\Gamma(\mathcal{H})$ and $\check{X}\in\Gamma(\mathcal{V})$.

 One can define the full curvature  
 of $\overline{\nabla}$ by the following formula:
\begin{eqnarray}
 \Phi(\xi,\eta,X,Y)&=&g(\phi(X,Y)\xi,\eta)\nonumber\\
                   &=&g(\phi(\hat{X},\hat{Y})\xi+\phi(\hat{X},\check{Y})\xi
                   +\phi(\check{X},\hat{Y})\xi+\phi(\check{X},\check{Y})\xi,\eta)\nonumber\\
                   &=&\textbf{R}(\xi,\eta,X,Y)+\textbf{P}(\xi,\eta,X,Y)+\textbf{Q}(\xi,\eta,X,Y),\nonumber
\end{eqnarray}
where 
  $
      \textbf{R}(\xi,\eta,X,Y)=g(\phi(\hat{X},\hat{Y})\xi,\eta),$ 
   $\textbf{P}(\xi,\eta,X,Y)=g(\phi(\hat{X},\check{Y})\xi,\eta)
                   +g(\phi(\check{X},\hat{Y})\xi,\eta)$ and $\textbf{Q}(\xi,\eta,X,Y)=g(\phi(\check{X},\check{Y})\xi,\eta)$
                   are respectively the \textit{first (horizontal) curvature}, 
\textit{mixed curvature} and \textit{vertical} curvature. 

In particular, if $\overline{\nabla}$ is the Chern connection, the $\textbf{Q}$-curvature vanishes.
\subsection{$R$-Einstein metric}
With respect to the Chern connection, we have:
\begin{definition}\label{defi4b}
\begin{itemize}
 \item [(1)]
   The horizontal Ricci tensor $\textbf{Ric}_F^H$ and the horizontal scalar curvature $\textbf{Scal}_F^H$ of $(M,F)$ 
   are respectively defined by
   $$\textbf{Ric}_F^H(\xi,X):=\sum_{a=1}^n\textbf{R}(\xi,e_a,X,\hat{e}_a)$$ and 
  $$\textbf{Scal}_F^H:=\sum_{a=1}^n\textbf{Ric}_F^H(e_a,\hat{e}_a)=\sum_{a,b=1}^n\textbf{R}(e_a,e_b,\hat{e}_a,\hat{e}_b).$$ 
     \item [(2)]A Finslerian metric $F$ on an $n$-dimensional manifold is
$R$-Einstein if   
  \begin{eqnarray}
   \textbf{Ric}_F^H=\frac{1}{n}\textbf{Scal}_F^H\underline{g}.\label{EinstC1}
  \end{eqnarray}
  \end{itemize}
  \end{definition}
  \begin{rem}\label{rem10}
 If $F$ satisfies (\ref{EinstC1})
 for a constant function $\textbf{Scal}_F^H$ (respectively for $\textbf{Scal}_F^H\equiv0$) then $F$ is said to be  
   horizontally Ricci-constant (respectively, $F$ is called horizontally Ricci-flat metric).
 \end{rem}
\begin{lemma}\cite{biNibaruta4}\label{proShu}
   If $F$ is horizontally an Einstein metric on a connected manifold of dimension $n\geq3$ then its horizontal scalar curvature 
   $\textbf{Scal}_F^H$ is constant.
  \end{lemma}
 \section{Finslerian $R$-Einstein conformal invariance}\label{adgbjah}
 \begin{definition}
Let $F$ and $\widetilde{F}$ be two conformal Finslerian metrics on a manifold $M$, with $\widetilde{F}=e^uF$
  . A geometric object $\mathcal{O}$ associated with 
  $F$ is said to be conformally invariant (respectively conformally $u$-invariant) 
  if the object $\widetilde{\mathcal{O}}$ 
  associated with $\widetilde{F}$ satisfies $\widetilde{\mathcal{O}}=\mathcal{O}$.
  (respectively, $\widetilde{\mathcal{O}}=e^u\mathcal{O}$).
\end{definition}
 \begin{pro}\label{lemgrosb}\cite{biNibaruta1}
  Let $F$ and $\widetilde{F}$ be two Finslerian metrics on an $n$-dimensional manifold $M$. 
  If $F$ is conformal to $\widetilde{F}$, with $\widetilde{F}=e^u F$,
 then the trace-free horizontal Ricci tensors $\textbf{E}_F^H$ and $\widetilde{\textbf{E}}_{\widetilde{F}}^H$, associated
   with $F$ and $\widetilde{F}$ respectively, are related by
\begin{eqnarray}\label{122b}
\widetilde{\textbf{E}}_{\widetilde{F}}^H 
= \textbf{E}_F^H
 -(n-2)\left(H_u-du\circ du\right)
-\frac{(n-2)}{n}\left(\Delta^Hu+||\triangledown u||_g^2\right)\underline{g}+{\varPsi}_u
\end{eqnarray}
where ${\varPsi}_u$ is the $(1,1;0)$-tensor on $(M,F)$ given by
\begin{eqnarray}\label{123b}
{\varPsi}_u(\xi,X)
&:=&(n-4)\mathcal{A}(\mathcal{B}(\textbf{h}(\triangledown u),\pi_*X,\xi))\nonumber\\
&&+(2-n)\Big[\mathcal{A}(\triangledown u,\mathcal{B}(X),\xi)
 +\mathcal{A}(\triangledown u,\pi_*X,\mathcal{B}(\textbf{h}(\xi)))\Big]\nonumber\\
  &&
  +\frac{1}{n}g^{ij}\left[2(n-2)\mathcal{A}(\triangledown u,\partial_i,\mathcal{B}(\hat{\partial}_j)))
  -3\mathcal{A}(\mathcal{B}(\textbf{h}(\triangledown u),\partial_j,\partial_i))\right]g(\xi,\pi_*X)\nonumber\\
    &&+g^{ij}\Big[g\left(\Theta(X,\textbf{h}(\Theta(\hat{\partial}_j,\textbf{h}(\xi)))),\partial_i\right)
  -g\left(\Theta(\hat{\partial}_j,\textbf{h}(\Theta(X,\textbf{h}(\xi))),\partial_i\right)\Big]\nonumber\\
  &&+g^{ij}\Big[g\left((\nabla_X\Theta) (\hat{\partial}_j,\textbf{h}(\xi)),\partial_i\right)
  -g\left((\nabla_j\Theta) (\textbf{h}(\xi),X),\partial_i\right)\Big]\nonumber\\
  &&
-\frac{1}{n}g^{ij}g^{kl}\Big[
  \mathcal{A}(\mathcal{B}(\textbf{h}(\Theta_{jk}),\partial_l,\partial_i))
-\mathcal{A}(\mathcal{B}(\textbf{h}(\Theta_{kl}),\partial_j,\partial_i))
  \Big]g(\xi,\pi_*X)\nonumber\\
  &&-\frac{1}{n}g^{ij}g^{kl}\Big[g\left((\nabla_l\Theta)_{jk},\partial_i\right)
  -g\left((\nabla_j\Theta)_{kl},\partial_i\right)\Big]g(\xi,\pi_*X),
\end{eqnarray}
 for every $\xi\in\Gamma(\pi^*TM)$ and $X\in\chi(\mathring{T}M)$ with $\Theta_{ij}=\Theta(\hat{\partial}_i,\hat{\partial}_j)$ 
 and 
$\mathcal{B}$ is the application which maps $\pi^*TM$ to $\pi^*TM$ given by 
\begin{eqnarray}\label{djgdhdgkgd1}
 \mathcal{B}=\mathcal{B}_j^i\partial_i\otimes dx^j 
\end{eqnarray}
with  
 \begin{eqnarray}
\mathcal{B}_j^i=\frac{1}{2F}\left(\nabla_ru\right)\frac{\partial(F^2g^{ir}-2y^iy^r)}{\partial y^j}.
 \end{eqnarray}
 \end{pro}

  By the Proposition \ref{lemgrosb}, if $F$ and $\widetilde{F}$ are conformal 
  then the $R$-Einstein criterions $\textbf{E}_F^H$ and $\widetilde{\textbf{E}}_{\widetilde{F}}^H$ satisfy
   the relation (\ref{122b}).
  In a local chart, if $\textbf{E}_F^H=\widetilde{\textbf{E}}_{\widetilde{F}}^H$ then we have
   \begin{eqnarray}\label{005b1q}
  0&=& \Big[
 -(n-2)\left(H_u-du\circ du\right)
-\frac{(n-2)}{n}\left(\Delta^Hu+||\triangledown u||_g^2\right)\underline{g}\Big](\partial_i,\hat{\partial}_j)\nonumber\\
&&+(2-n)\Big[\mathcal{A}(\triangledown u,\mathcal{B}(\hat{\partial}_j),\partial_i)
 +\mathcal{A}(\triangledown u,\pi_*\hat{\partial}_j,\mathcal{B}(\textbf{h}(\partial_i)))\Big]\nonumber\\
  &&+(n-4)\mathcal{A}(\mathcal{B}(\textbf{h}(\triangledown u),\pi_*\hat{\partial}_j,\partial_i))\nonumber\\
&&
  +\frac{1}{n}g^{kl}\Big[2(n-2)\mathcal{A}(\triangledown u,\partial_k,\mathcal{B}(\hat{\partial}_l)))
  -3\mathcal{A}(\mathcal{B}(\textbf{h}(\triangledown u),
  \partial_l,\partial_k))\Big]g(\partial_i,\pi_*\hat{\partial}_j).\nonumber\\
    &&+g^{ij}\Big[g\left(\Theta(\hat{\partial}_j,\textbf{h}(\Theta(\hat{\partial}_l,\textbf{h}(\partial_i)))),\partial_k\right)
  -g\left(\Theta(\hat{\partial}_j,\textbf{h}(\Theta(\hat{\partial}_l,\textbf{h}(\partial_i))),\partial_k\right)\Big]\nonumber\\
  &&+g^{kl}\Big[g\left((\nabla_j\Theta) (\hat{\partial}_l,\textbf{h}(\partial_i)),\partial_k\right)
  -g\left((\nabla_l\Theta) (\textbf{h}(\partial_i),\hat{\partial}_j),\partial_k\right)\Big]\nonumber\\
  &&
-\frac{1}{n}g^{rs}g^{kl}\Big[
  \mathcal{A}(\mathcal{B}(\textbf{h}(\Theta_{sk}),\partial_l,\partial_r))
-\mathcal{A}(\mathcal{B}(\textbf{h}(\Theta_{kl}),\partial_s,\partial_r))
  \Big]g_{ij}\nonumber\\
  &&-\frac{1}{n}g^{rs}g^{kl}\left[g\left((\nabla_l\Theta)_{sk},\partial_r\right)
  -g\left((\nabla_s\Theta)_{kl},\partial_r\right)\right]g_{ij}.
 \end{eqnarray}
  
  By (\ref{djgdhdgkgd1}), we have $ \mathcal{B}(\hat{\partial}_l)=\mathcal{B}_l^{s_1}\partial_{s_1}$ 
where $\mathcal{B}_l^{s_1}=\frac{1}{2F}\left(\nabla_ru\right)\frac{\partial(F^2g^{{s_1}r}-2y^{s_1}y^r)}{\partial y^l}$ and
 $\mathcal{B}(\textbf{h}(\triangledown u))=\nabla^lu\mathcal{B}_l^{s_1}\partial_{s_1}.$
Thus, from (\ref{005b1q}), we have
 \begin{eqnarray}
  I_{11}&=&(2-n)\left[\mathcal{A}(\triangledown u,\mathcal{B}(\hat{\partial}_j),\partial_i)
 +\mathcal{A}(\triangledown u,\pi_*\hat{\partial}_j,\mathcal{B}(\textbf{h}(\partial_i)))\right]\nonumber\\
  &&+(n-4)\mathcal{A}(\mathcal{B}(\textbf{h}(\triangledown u),\pi_*\hat{\partial}_j,\partial_i))\nonumber\\
&=&(n-4)\nabla^{s_2}u\mathcal{B}_{s_2}^{s_1}\mathcal{A}_{s_1ij}
-(n-2)\Big(\nabla^{s_2}u\mathcal{B}_{i}^{s_1}\mathcal{A}_{s_1js_2}
+\nabla^{s_2}u\mathcal{B}_{j}^{s_1}\mathcal{A}_{s_1is_2}\Big),\nonumber\\
I_{12}&=&\frac{1}{n}g^{kl}\Big[2(n-2)\mathcal{A}(\triangledown u,\partial_k,\mathcal{B}(\hat{\partial}_l)))
  -3\mathcal{A}(\mathcal{B}(\textbf{h}(\triangledown u),
  \partial_l,\partial_k))\Big]g_{ij}\nonumber\\
   &=&-\frac{1}{n}g^{kl}\nabla^{s_2}u\Big[-3\mathcal{B}_{s_2}{s_1}\mathcal{A}_{s_1kl}
   +3\mathcal{B}_{k}{s_1}\mathcal{A}_{s_1ls_2}-(2n-1)\mathcal{B}_{k}{s_1}\mathcal{A}_{s_1ls_2}\Big]g_{ij},\nonumber\\
   I_{13}&=&-\frac{1}{n}g^{rs}g^{kl}\left[
  \mathcal{A}(\mathcal{B}(\textbf{h}(\Theta_{sk}),\partial_l,\partial_r))
-\mathcal{A}(\mathcal{B}(\textbf{h}(\Theta_{kl}),\partial_s,\partial_r))
  \right]g_{ij},\nonumber\\
       I_{14}&=&g^{kl}\left[g\left(\Theta(\hat{\partial}_j,\textbf{h}(\Theta(\hat{\partial}_l,\textbf{h}(\partial_i)))),\partial_k\right)
  -g\left(\Theta(\hat{\partial}_l,\textbf{h}(\Theta(\hat{\partial}_j,\textbf{h}(\partial_i))),\partial_k\right)\right]\nonumber\\
  &=&\frac{1}{n}g^{kl}g^{rs}g_{ij}\left[g\left(\Theta(\hat{\partial}_s,\textbf{h}
  (\Theta(\hat{\partial}_l,\textbf{h}(\partial_r)))),\partial_k\right)
  -g\left(\Theta(\hat{\partial}_l,\textbf{h}(\Theta(
  \hat{\partial}_s,\textbf{h}(\partial_r))),\partial_k\right)\right]\nonumber\\
   &=&-I_{13},\nonumber\\
    I_{15}&=&-\frac{1}{n}g^{rs}g^{kl}\left[g\left((\nabla_l\Theta)_{sk},\partial_r\right)
   -g\left((\nabla_s\Theta)_{kl},\partial_r\right)\right]g_{ij},\nonumber
     \end{eqnarray}
  \begin{eqnarray}
 I_{16}
 &=&-\frac{1}{n}g^{rs}g^{kl}\left[g\left((\nabla_l\Theta)_{sk},\partial_r\right)
  -g\left((\nabla_s\Theta)_{kl},\partial_r\right)\right]g_{ij}\nonumber\\
  &=&g^{kl}[\frac{1}{n}g_{t_1t_2}g^{t_1t_2}][g_{it_1}g^{rt_1}g_{jt_2}g^{st_2}]
 \left[g\left((\nabla_s\Theta)_{lr},\partial_k\right)
  -g\left((\nabla_l\Theta)_{rs},\partial_k\right)\right]\nonumber\\
       &=&-I_{15}.\nonumber
 \end{eqnarray}
  Hence, putting the expressions of $I_{11},I_{12},I_{13},I_{14},I_{15}$ and $I_{16}$ 
in the right-hand side of (\ref{005b1q}) we obtain 
the equation in (\ref{E-E9}) given 
%
%
 in the following Proposition.
 \begin{pro}\label{cor10'}
Let $F$ and $\widetilde{F}$ be two conformal Finslerian metrics on an $n$-dimensional manifold $M$, 
  with $\widetilde{F}=e^uF$. Then $(M,F)$ and $(M,\widetilde{F})$ have locally a same $R$-Einstein criterion
   if and only if
  \begin{eqnarray}\label{E-E9}
  \nabla_j\nabla_iu&=&\frac{1}{n}\left(\nabla^d\nabla_du-\nabla^du\nabla_du\right)g_{ij}+\nabla_iu\nabla_ju\nonumber\\
                  &&+\frac{(n-1)}{2n(n-2)F}\left(\nabla_ru\nabla^{q}u\right)\frac{\partial(F^2g^{rs}-2y^ry^s)}{\partial y^q}
                  g^{kl}\mathcal{A}_{skl}g_{ij}.
         \end{eqnarray}
 \end{pro} 
        \subsection{Warped product of Finslerian metrics}\label{section3.2}

Let $\stackrel{\oldstylenums{1}}{M}$ and $\stackrel{\oldstylenums{2}}{M}$ be two $C^{\infty}$ manifolds. 
For every $(x_1,x_2)\in \stackrel{\oldstylenums{1}}{M}\times \stackrel{\oldstylenums{2}}{M}$, we have the following properties.
\begin{itemize}
 \item [(1)] The projections 
 \begin{eqnarray}
  \stackrel{\oldstylenums{1}}{p}&:&\stackrel{\oldstylenums{1}}{M}\times \stackrel{\oldstylenums{2}}{M}\longrightarrow \stackrel{\oldstylenums{1}}{M}
  \text{		such that		} \stackrel{\oldstylenums{1}}{p}(x_1,x_2)=x_1\nonumber\\
   \stackrel{\oldstylenums{2}}{p}&:&\stackrel{\oldstylenums{1}}{M}\times \stackrel{\oldstylenums{2}}{M}\longrightarrow \stackrel{\oldstylenums{1}}{M}
  \text{		such that		} \stackrel{\oldstylenums{2}}{p}(x_1,x_2)=x_2\nonumber
 \end{eqnarray}
are $C^{\infty}$ submersions.
  \item[(2)] $dim(\stackrel{\oldstylenums{1}}{M}\times \stackrel{\oldstylenums{2}}{M})=dim 
  \stackrel{\oldstylenums{1}}{M}+ dim \stackrel{\oldstylenums{2}}{M}$.
\end{itemize}

The warped product manifold of two Finslerian manifolds is defined as follows. 
\begin{definition}\label{defi11c}
Let $(\stackrel{\oldstylenums{1}}{M},\stackrel{\oldstylenums{1}}{F})$ and
 $(\stackrel{\oldstylenums{2}}{M},\stackrel{\oldstylenums{2}}{F})$ be two Finslerian manifolds. 
 Let $f$ be a positive $C^{\infty}$ function on $\stackrel{\oldstylenums{1}}{M}$. 
 The warped product of $(\stackrel{\oldstylenums{1}}{M},\stackrel{\oldstylenums{1}}{F})$ and
 $(\stackrel{\oldstylenums{2}}{M},\stackrel{\oldstylenums{2}}{F})$ is a manifold 
 $M=\stackrel{\oldstylenums{1}}{M}\times_f \stackrel{\oldstylenums{2}}{M}$ equipped with the Finslerian metric 
 \begin{eqnarray}\label{11c1}
  F:\mathring{T}\stackrel{\oldstylenums{1}}{M}\times \mathring{T}\stackrel{\oldstylenums{2}}{M}\longrightarrow \mathbb{R}^+
 \end{eqnarray}
 such that for any vector tangent $y\in T_xM$, 
 with $x=(x_1,x_2)\in M$ and $y=(y_1,y_2)$,
 \begin{eqnarray}\label{11c}
  F(x,y)=\sqrt{\stackrel{\oldstylenums{1}}{F^2}(x_1,\stackrel{\oldstylenums{1}}{p}_*y)+
  f^2(\stackrel{\oldstylenums{1}}{p}(x_1,x_2))\stackrel{\oldstylenums{2}}{F^2}(x_2,\stackrel{\oldstylenums{2}}{p}_*y)}
 \end{eqnarray}
  where 
 $\stackrel{\oldstylenums{1}}{p}$ 
 and $\stackrel{\oldstylenums{2}}{p}$ are respectively the projections of $\stackrel{\oldstylenums{1}}{M}\times \stackrel{\oldstylenums{2}}{M}$ 
 onto $\stackrel{\oldstylenums{1}}{M}$ and 
 $\stackrel{\oldstylenums{2}}{M}$. 
\end{definition}

The function $F$ defined in (\ref{11c1}) and (\ref{11c}) is a Finslerian manifold. More precisely,
\begin{itemize}
 \item[(i)]$F$ is $C^{\infty}$ on 
 $\mathring{T}\stackrel{\oldstylenums{1}}{M}\times \mathring{T}\stackrel{\oldstylenums{2}}{M}$ since $\stackrel{\oldstylenums{1}}{F}$
 and $\stackrel{\oldstylenums{2}}{F}$ are respectively 
 $C^{\infty}$ on $\mathring{T}\stackrel{\oldstylenums{1}}{M}$ and $\mathring{T}\stackrel{\oldstylenums{2}}{M}$.
 \item [(ii)]$F$ is homogeneous of degree $1$ in $y=(y_1,y_2)\in T_xM$.
 \item [(iii)]If $n_1$ and $n_2$ are respectively the dimensions of $(\stackrel{\oldstylenums{1}}{M},\stackrel{\oldstylenums{1}}{F})$
 and $(\stackrel{\oldstylenums{2}}{M},\stackrel{\oldstylenums{2}}{F})$,
  each element of the Hessian matrix $(g_{ij}(x,y))_{1\leq i,j\leq n_1+n_2}$ of $\frac{1}{2}F^2$,
  has the following form:
 \begin{eqnarray}\label{}
  g_{ij}(x,y)&:=&\frac{\partial^2\left[\frac{1}{2}F^2(x,y)\right]}{\partial y^i\partial y^j}\nonumber\\
   &=&\frac{1}{2}\frac{\partial^2\stackrel{\oldstylenums{1}}{F^2}(x_1,y_1)}{\partial y_1^i\partial y_1^j}
   +\frac{1}{2}f^2(x_1)\frac{\partial^2\stackrel{\oldstylenums{2}}{F^2}(x_2,y_2)}{\partial y_2^i\partial y_2^j}.\nonumber
 \end{eqnarray}
 for every point $(x,y)=(x_1,x_2,y_1,y_2)\in \mathring{T}\stackrel{\oldstylenums{1}}{M}\times \mathring{T}\stackrel{\oldstylenums{2}}{M}$.
 Thus, 
  \begin{eqnarray}\label{31cc}
\big(g_{ij}(x,y)\big)=
\left(\begin{array}{cc}
\big(\stackrel{\oldstylenums{1}}{g}_{ij}(x_1,y_1)\big) & 0 \\
 0  & \big(\stackrel{\oldstylenums{2}}{g}_{ij}(x_2,y_2)\big)
 \end{array} \right)
 \end{eqnarray}
 where $\stackrel{\oldstylenums{1}}{g}_{ij}(x_1,y_1):
 =\frac{1}{2}\frac{\partial^2\stackrel{\oldstylenums{1}}{F^2}(x_1,y_1)}{\partial y_1^i\partial y_1^j}$ and 
 $\stackrel{\oldstylenums{2}}{g}_{ij}(x_2,y_2):=
 \frac{1}{2}f^2(x_1)\frac{\partial^2\stackrel{\oldstylenums{2}}{F^2}(x_2,y_2)}{\partial y_2^i\partial y_2^j}.$ 
 So the Hessian matrix of $F$ is positive definite at every point
 $(x_1,x_2,y_1,y_2)\in \mathring{T}\stackrel{\oldstylenums{1}}{M}\times \mathring{T}\stackrel{\oldstylenums{2}}{M}$ 
 because the Hessian matrices of $\stackrel{\oldstylenums{1}}{F}$ and $\stackrel{\oldstylenums{2}}{F}$ are.
\end{itemize}
\subsection{Curvatures associated with warped product Finslerian metrics} 
\begin{pro}\label{pro2.1}
 Let $(\stackrel{\oldstylenums{1}}{M},\stackrel{\oldstylenums{1}}{F})$ and
 $(\stackrel{\oldstylenums{2}}{M},\stackrel{\oldstylenums{2}}{F})$ be two Finslerian manifolds. 
 On a warped product manifold $M=\stackrel{\oldstylenums{1}}{M}\times_f\stackrel{\oldstylenums{2}}{M}$, if 
 $\stackrel{\oldstylenums{1}}{\xi}\in\Gamma(\stackrel{\oldstylenums{1}}{\pi^*}T\stackrel{\oldstylenums{1}}{M})$, 
 $\stackrel{\oldstylenums{2}}{\xi}\in\Gamma(\stackrel{\oldstylenums{2}}{\pi^*}T\stackrel{\oldstylenums{2}}{M})$ and 
  $\stackrel{\oldstylenums{1}}{X}\in\chi(\mathring{T}\stackrel{\oldstylenums{1}}{M})$ then 
  \begin{itemize}
   \item [(i)]$\nabla_{\stackrel{\oldstylenums{1}}{X}}\stackrel{\oldstylenums{1}}{\xi}
   =\stackrel{\oldstylenums{1}}{\nabla}_{\stackrel{\oldstylenums{1}}{X}}\stackrel{\oldstylenums{1}}{\xi}$ where 
   $\stackrel{\oldstylenums{1}}{\nabla}$ is the Chern connection associated with $(\stackrel{\oldstylenums{1}}{M},\stackrel{\oldstylenums{1}}{F})$.
   \item [(ii)]$\nabla_{\stackrel{\oldstylenums{1}}{X}}\stackrel{\oldstylenums{2}}{\xi}
   =\frac{1}{f}\stackrel{\oldstylenums{1}}{X}(f)\stackrel{\oldstylenums{2}}{\xi}$.
  \end{itemize}
\end{pro}
As a direct consequence, we have
\begin{cor}\label{propdgeh}
 Let $(\stackrel{\oldstylenums{1}}{M},\stackrel{\oldstylenums{1}}{F})$ and
 $(\stackrel{\oldstylenums{2}}{M},\stackrel{\oldstylenums{2}}{F})$ be two Finslerian manifolds. 
 On a warped product manifold $M=\stackrel{\oldstylenums{1}}{M}\times_f\stackrel{\oldstylenums{2}}{M}$, if 
 $\stackrel{\oldstylenums{1}}{\xi},\stackrel{\oldstylenums{1}}{\eta}\in\Gamma(\stackrel{\oldstylenums{1}}{\pi^*}T\stackrel{\oldstylenums{1}}{M})$, 
  $\stackrel{\oldstylenums{1}}{X}, \stackrel{\oldstylenums{1}}{Y}\in\chi(\mathring{T}\stackrel{\oldstylenums{1}}{M})$ and
  $\stackrel{\oldstylenums{2}}{X}\in\chi(\mathring{T}\stackrel{\oldstylenums{2}}{M})$ then 
  \begin{itemize}
   \item [(i)]$\textbf{R}(\stackrel{\oldstylenums{1}}{\xi},\stackrel{\oldstylenums{1}}{\eta},
   \stackrel{\oldstylenums{1}}{X}, \stackrel{\oldstylenums{1}}{Y})
   =\stackrel{\oldstylenums{1}}{\textbf{R}}(\stackrel{\oldstylenums{1}}{\xi},\stackrel{\oldstylenums{1}}{\eta},
   \stackrel{\oldstylenums{1}}{X}, \stackrel{\oldstylenums{1}}{Y})$.
   \item [(ii)]$\textbf{R}(\stackrel{\oldstylenums{1}}{\xi},\stackrel{\oldstylenums{1}}{\eta},
   \stackrel{\oldstylenums{2}}{X}, \stackrel{\oldstylenums{1}}{Y})=0$.
   \end{itemize}
\end{cor}
\subsection{Main results}
We prove the following.
\begin{theo}\label{klwlefh}
   Let $F$ and $\widetilde{F}$ be two Finslerian metrics 
   on a manifold $M$ of dimension $n\geq3$. A conformal deformation $\widetilde{F}$ of $F$, with $\widetilde{F}=\varphi^{-1}F$, 
   preserves the $R$-Einstein criterion if and only if :
   \begin{itemize}
    \item [(1)] $\varphi$ is constant if $M$ is (locally) closed with $\widetilde{\textbf{Ric}}_{\widetilde{F}}^H=\textbf{Ric}_F^H$.
    \item[(2)]$\varphi$ is everywhere non-constant 
in a neighborhood $U$ of a point $x\in M$ if $(M,\varphi^{-1}F)$ is a Finslerian cylinder of the form 
$\Big((0,\varepsilon)\times \stackrel{\oldstylenums{2}}{M},\sqrt{t^2+(\varphi^{'}(t))^2\stackrel{\oldstylenums{2}}{F}^2}\Big)$,
with $\varphi$ depending only on $t\in (0,\varepsilon)$.
   \end{itemize}
   \end{theo}
 \begin{proof}
  Let $\widetilde{F}:=e^uF$ with $e^u=\varphi^{-1}$ 
 be a conformal deformation of $F$. We can show that, for the conformal factor $\varphi$, the equation (\ref{E-E9}) takes the form 
 $\nabla_j\nabla_i\varphi= fg_{ij}$ 
 for some $f\in C^{\infty}(\mathring{T}M, \mathbb{R})$. Precisely, we have
\begin{eqnarray}\label{E-E10c2}
  \nabla_j\nabla_i{\varphi}
  &=&\frac{1}{n}\Big[ \nabla^d\nabla_d{\varphi}
                  -\frac{(n-1)}{2(n-2) F}\left(\nabla_r\varphi\nabla^{q}\varphi\right)\frac{\partial(F^2g^{rs}-2y^ry^s)}{\partial y^q}
                  g^{kl}\mathcal{A}_{skl}\Big]g_{ij}\nonumber\\
                 &=& fg_{ij}\label{192019}
         \end{eqnarray}
 where $f:=\frac{1}{n}\Big[ \nabla^d\nabla_d{\varphi}
                  -\frac{(n-1)}{2(n-2) F}\left(\nabla_r\varphi\nabla^{q}\varphi\right)\frac{\partial(F^2g^{rs}-2y^ry^s)}{\partial y^q}
                  g^{kl}\mathcal{A}_{skl}\Big]$.\\

 $(1)$ If $\widetilde{F}=\varphi^{-1}F$ on $U\subseteq M$ closed and $\widetilde{\textbf{Ric}}_{\widetilde{F}}^H=\textbf{Ric}_F^H$ 
 then $\widetilde{\textbf{E}}_{\widetilde{F}}^H=\textbf{E}_F^H$. 
 As shown in \cite{biNibaruta1}, $\varphi$ is constant. 
 
 Conversely, set $\varphi=e^{-u}$.
 If $\varphi$ is constant then from the Proposition \ref{lemgrosb}, 
 $\widetilde{\textbf{E}}_{\widetilde{F}}^H=\textbf{E}_F^H$.

 $(2)$ Define $\varphi:(0,\varepsilon)\times \stackrel{\oldstylenums{2}}{M}\longrightarrow (0,\infty)$ by 
   $\varphi(t,p)=\varphi(t)$. Then 
   \begin{eqnarray}\label{104}
    \triangledown\varphi=\nabla_t\varphi\partial_t
   \end{eqnarray}
 and 
 \begin{eqnarray}\label{105}
   \nabla_t\triangledown\varphi\stackrel{(\ref{104})}{=}
   \nabla_t\nabla_t\varphi\partial_t+\nabla_t\varphi\nabla_t\partial_t=\ddot{\varphi}\partial_t.
   \end{eqnarray} 
 The relation (\ref{105}) shows that $\nabla_t\nabla_t\varphi=\ddot{\varphi}(t)g_{tt}$.\\
 
 Now we show that $\nabla_{\alpha}\nabla_{\beta}\varphi= f\stackrel{\oldstylenums{2}}{g}_{\alpha\beta}$ 
 for $\alpha,\beta=1,...,n-1$. We have
 \begin{eqnarray}
  \nabla_{\alpha}\triangledown\varphi&\stackrel{(\ref{104})}{=}&\nabla_{\alpha}\dot{\varphi}(t)\partial_t\nonumber\\
    &=&\dot{\varphi}\nabla_{\alpha}\partial_t\nonumber\\
    &=&\dot{\varphi}\sum_{i=1}^n\Gamma_{\alpha n}^i\partial_i, ~~~t=t^n\nonumber\\
    &=&\dot{\varphi}\sum_{i=1}^n\left\{\frac{1}{2}\sum_{l=1}^n\left[g^{il}
    \left(\frac{\delta g_{\alpha l}}{\delta x^n}+\frac{\delta g_{n l}}{\delta x^{\alpha}}
    -\frac{\delta g_{\alpha n}}{\delta x^l}\right)\partial_i\right]\right\}\nonumber\\
    &=&\frac{1}{2}\dot{\varphi}\sum_{l=1}^n\left\{\left[g^{ln}
    \left(\frac{\delta g_{\alpha l}}{\delta x^n}+\frac{\delta g_{n l}}{\delta x^{\alpha}}
    -\frac{\delta g_{\alpha n}}{\delta x^l}\right)\partial_n\right]\right.\nonumber\\
    &&\left.+\sum_{\beta=1}^{n-1}\left[g^{l \beta}
    \left(\frac{\delta g_{\alpha l}}{\delta x^n}+\frac{\delta g_{n l}}{\delta x^{\alpha}}
    -\frac{\delta g_{\alpha n}}{\delta x^l}\right)\partial_{\beta}\right]\right\}.\label{106}
 \end{eqnarray}
 Since $g_{\alpha n}=0=\frac{\delta g_{n l}}{\delta x^{\alpha}}$, we obtain from relation (\ref{106})
 \begin{eqnarray}
  \nabla_{\alpha}\triangledown\varphi&=&\frac{1}{2}\dot{\varphi}\sum_{l=1}^n
  \left[g^{l\beta}\left(\frac{\delta g_{\alpha l}}{\delta x^n}\right)\partial_{\beta}\right]\nonumber\\
 &=&\frac{1}{2}\dot{\varphi}\left\{\frac{1}{\dot{\varphi}^2}
 \stackrel{\oldstylenums{2}}{g}^{l \beta}\left[\frac{\partial (\dot{\varphi}^2\stackrel
 {\oldstylenums{2}}{g}_{\alpha l})}{\partial t}\right]\partial_{\beta}\right\}\nonumber\\
 &=&\ddot{\varphi}\partial_{\alpha}.
 \end{eqnarray}
 That is $\nabla_{\alpha}\nabla_{\beta}\varphi=\ddot{\varphi}\stackrel{\oldstylenums{2}}{g}_{\alpha\beta}, \alpha,\beta=1,...,n$.
 
 Conversely, if $\varphi$ is everywhere non-constant on $(M,F)$ and if $(M,\widetilde{F})$ is a Finslerian cylinder of the form 
$\Big((0,\varepsilon)\times \stackrel{\oldstylenums{2}}{M},\sqrt{t^2+(\varphi^{'}(t))^2\stackrel{\oldstylenums{2}}{F}^2}\Big)$ 
then the equation (\ref{192019}) holds. It follows from the Proposition \ref{lemgrosb} and by setting $e^u=\varphi^{-1}$ that 
 $\widetilde{\textbf{E}}_{\widetilde{F}}^H=\textbf{E}_F^H$.
 \end{proof}
 Thus, we claim.
 \begin{cor}\label{uieleied} A Finslerian $R$-Einstein space $(M,F)$ is locally conformal to an other $R$-Einstein space 
 $(M,\widetilde{F})$, with $\widetilde{F}=\varphi^{-1}F$, if and only if
   \begin{itemize}
    \item [(1)] the conformal factor is constant if $M$ is (locally) closed with $\widetilde{\textbf{Ric}}_{\widetilde{F}}^H=\textbf{Ric}_F^H$.
    \item[(2)]the conformal factor $\varphi$ is everywhere non-constant 
in a neighborhood $U$ of a point $x\in M$ if $(M,\varphi^{-1}F)$ has the form 
$\Big((0,\varepsilon)\times \stackrel{\oldstylenums{2}}{M},\sqrt{t^2+(\varphi^{'}(t))^2\stackrel{\oldstylenums{2}}{F}^2}\Big)$,
with $\varphi$ depending only on $t\in (0,\varepsilon)$.
   \end{itemize}
   \end{cor}
\begin{rem}
 The last corollary is a particular case of the Theorem \ref{klwlefh}.
\end{rem}
\section{Example}\label{section3.gg}
 Let $\varphi:(0,\pi)\longrightarrow (0,\infty)$ be a $C^{\infty}$ map such that $\varphi(t)=cos t+c$ with $c\in (1,\infty)$
 and, $\mathbb{S}^3$ and $\mathbb{S}^2$ 
 the unit spheres. Consider the warped product 
 $M=(0,\pi)\times_{\varphi'} \mathbb{S}^2$ and the map $i:(M,F)\longrightarrow (\mathbb{S}^3,F_0)$ 
 defined by $i(t,z^1,z^2)=(cost,z^1sint,z^2sint)$. We can show that $i$ is a diffeomorphism and a local isometry. 
 For $y=(y^1,y^2,y^3)\in\mathbb{S}^3$ and $z=(z^1,z^2)\in\mathbb{S}^2$, we have 
 $y=(cost,z^1sint,z^2sint)$. 
 Then, 
    \begin{equation}
   \left \{
    \begin{array}{r c l}
      dy^1&=&-sintdt\nonumber\\
    dy^2&=&z^1costdt+sintdz^1\nonumber\\
   dy^3&=&z^2costdt+sintdz^2.\nonumber
    \end{array}
    \right.
\end{equation}
The fundamental tensor associated with $F_0$ is 
 \begin{eqnarray}
  g_0(y)&=&\delta_{ij}dy^idy^j, \text{   with   } i,j=1,2,3\nonumber\\
  &=&sin^2tdt^2+(z^1)^2cos^2tdt^2+sin^2td{z^1}^2+2z^1costsintdtdz^1\nonumber\\
  &&+(z^2)^2cos^2tdt^2+sin^2td{z^2}^2+2z^1costsintdtdz^2\nonumber\\
  &=&sin^2tdt^2+\big((z^1)^2+(z^2)^2\big)cos^2tdt^2\nonumber\\
  &&+sin^2t\big(d{z^1}^2+d{z^1}^2\big)
  +2costsintdt\big(z^1dz^1+z^2dz^2\big)\nonumber\\
  &=&dt^2+sin^2t(d{z^1}^2+d{z^2}^2).\nonumber
 \end{eqnarray}
 Hence, by the formula (\ref{11c}), $F_0(y)=\sqrt{t^2+sin^2t\stackrel{\oldstylenums{2}}{F}^2}$ where 
 $\stackrel{\oldstylenums{2}}{F}$ is the Finslerian metric on $\mathbb{S}^2$.

 \end{document}